\documentclass[12pt,a4paper]{article}
\usepackage[utf8]{inputenc}
\usepackage[english]{babel}
\usepackage{amsmath}
\usepackage{amsfonts}
\usepackage{amssymb}
\usepackage{color}
\usepackage[left=2cm,right=2cm,top=2cm,bottom=2cm]{geometry}

\title{Analysis of Boundary-Domain Integral Equations based on a new parametrix for the mixed diffusion BVP with variable coefficient in an interior Lipschitz domain}
\author{S. E. Mikhailov, C.F. Portillo}
\date{ }

\newenvironment{proof}{\paragraph{Proof:}}{\hfill$\square$}

\newtheorem{theorem}{Theorem}[section]

\newtheorem{lemma}[theorem]{Lemma}
\newtheorem{corollary}[theorem]{Corollary}

\numberwithin{equation}{section}

\allowdisplaybreaks
\begin{document}
\maketitle

\begin{abstract}
A mixed boundary value problem for the partial differential equation of diffusion in an inhomogeneous medium in a Lipschitz domain is reduced to a system of direct segregated parametrix-based Boundary-Domain Integral Equations (BDIEs). We use a parametrix different from the one employed in the papers by Mikhailov {(2002, 2006) and Chkadua, Mikhailov, Natroshvili (2009)}. We prove the equivalence between the original BVP and the corresponding BDIE system. The invertibility and Fredholm properties of the boundary-domain integral operators are also analysed.
\end{abstract}

\section{Introduction}
 { Boundary Domain Integral Equations (BDIEs) associated with variable-coefficient PDEs were studied in
\cite{CMN-1} for a scalar mixed elliptic BVP in bounded domains, \cite{exterior, carlos3} for the corresponding problem in unbounded domains, \cite{carlos1} for the mixed problem for the incompressible Stokes system in bounded domains, \cite{carlos4, carlos5} for the mixed problem for the Compressible Stokes in bounded domains and \cite{dufera}} for a 2D mixed elliptic problem in bounded domains. Further results on the theory of BDIEs for BVPs with variable coefficient can be found on
\cite{carlos2, localised,united, traces,numerics,miksolandreg,dufera,ayele}. Let us note that these types of BVPs model, for example, the heat transfer in inhomogeneous media or the motion of a laminar fluid with variable viscosity. 

{ The BDIE systems can be solved numerically after discretising them e.g. by the collocation method, cf. \cite{MikMoh2012, RS2013, RS2014}, which leads to the linear algebraic systems with fully populated matrices. The method performance is essentially improved by implementing hierarchical matrix compression technique in conjunction with the adaptive cross approximation  procedure to  \cite{numerics} and iterative methods, cf. \cite{SHK1999}. 
Another option is to discretise the {\em localised} version of BDIEs, based on the localised parametrices, which leads to systems of linear algebraic equations with sparse matrices \cite{localised2,sladek, RS2013, RS2014}.  
}
 
In order to deduce a BDIE system for a BVP with variable coefficients, usually a parametrix (Levi function) 
strongly related with the fundamental solution of the corresponding PDE with constant coefficients is employed. Using this relation, it is possible to establish further relations between the surface and volume potential type operators for the variable-coefficient case with their counterparts for the constant coefficient case, see, e.g. \cite[Eq. (3.10)-(3.13)]{CMN-1}, \cite[Eq. (34.10)-(34.16)]{carlos1}. 

For  the scalar operator
\begin{align}\label{Adef}
Au(x) := \sum_{i=1}^{3}\dfrac{\partial}{\partial x_{i}}\left(a(x)\dfrac{\partial u(x)}{\partial x_{i}}\right),
\end{align}
a parametrics 
 $$P^{y}(x,y)= P(x,y;a(y))=\dfrac{-1}{4\pi a(y)\vert x-y\vert}$$
has been employed in \cite{CMN-1,miksolandreg,exterior}, { where $x$ is the integration variable in the parametrix-based integral potentials}. 
Note that the superscript in $P^{y}(x,y)$ means that the parametrix is expressed in terms of the variable coefficient at point $y$.

There are many different ways of { constructing} parametrices and corresponding parametrix-based potentials and BDIEs,
for the same variable-coefficient PDE{, and performance of the BDIE-based numerical methods essentially depends on the chosen parametrix. To optimise the numerical method, it is beneficial to analyse the BDIEs based on different parametrices.}  
It appeared, however, that not always the corresponding parametrix-based potentials and BDIEs can be easily analysed.
The main motivation of this paper is to extend the collection of parametrices for which the analysis of the parametrix-based potentials and BDIEs is tractable. 
This will then allow to chose the tractable parametrices with more preferable properties, e.g., for numerical implementation. To this end, we employ  \textit{in this paper} the parametrix
  $$P^{x}(x,y)= P(x,y;a(x))=\dfrac{-1}{4\pi a(x)\vert x-y\vert}$$
for the same operator $A$ defined by \eqref{Adef}, { where $x$ is again the integration variable in the parametrix-based integral potentials.

Different families of parametrices lead to different relations with their counterparts for the constant coefficient case.  For the parametrices considered in this paper these relations are rather simple, which makes it possible to obtain the mapping properties of the integral potentials in Sobolev spaces and prove the equivalence between the BDIE system and the BVP. 
After studying the Fredholm properties of the matrix operator which defines the systems, their invertibility is proved, which implies the uniqueness of solution of the BDIE system.}

\section{Preliminaries and the BVP}
{ Let $\Omega=\Omega^{+}$ be a bounded simply connected open Lipschitz domain, $\Omega^{-}:=\mathbb{R}^{3}\setminus\overline{\Omega^{+}}$ the complementary (unbounded) domain. 
The Lipschitz boundary $\partial\Omega$ is connected and closed. 
Furthermore, $\partial\Omega :=\overline{\partial\Omega_{N}}\cup \overline{\partial\Omega_{D}}$ where both $\partial\Omega_{N}$ and $\partial\Omega_{D}$ are non-empty, connected disjoint Lipschitz submanifolds of $\partial\Omega$ with a Lipschitz interface between them. }

Let us consider the partial differential equation 
\begin{equation}\label{ch4operatorA}
Au(x):=\sum_{i=1}^{3}\dfrac{\partial}{\partial x_{i}}\left(a(x)\dfrac{\partial u(x)}{\partial x_{i}}\right)=f(x),\,\,x\in \Omega,
\end{equation}
where the variable smooth coefficient $a(x)\in \mathcal{C}^{\infty}(\overline{\Omega})$ is such that
\begin{align}\label{a-cond}
0<a_{\rm min}\le a(x)\le a_{\rm max}<\infty,\quad\forall x\in \overline{\Omega},
\end{align}
$u(x)$ is an unknown function and $f$ is a given function on $\Omega$. It is easy to see that if $a\equiv 1$ then, the operator $A$ becomes the Laplace operator, $\Delta$.
 
 We will use the following function spaces in this paper (see e.g. \cite{lions, McLean2000} for more details). Let $\mathcal{D}'(\Omega)$ be the
Schwartz distribution space; $H^{s}(\Omega)$ and $H^{s}(\partial\Omega)$ with $s\in \mathbb{R}$ be
the Bessel potential spaces; the space $\widetilde H^{s}(\Omega)$ consisting of all the distributions of $H^{s}(\mathbb{R}^{3})$ whose support belongs to the closed set $\overline\Omega$. 
{ 
The corresponding spaces in $\Omega^-$ are defined similarly.
We will also need the following  spaces on a boundary subset $S_1$,
$
\widetilde{H}^{s}(S_1):= \overline{\mathcal{C}^{\infty}_{0}(S_1)}^{\parallel\cdot\parallel_{H^{s}(\mathbb{R}^{3})}},
$ 
which can be characterized as $\widetilde{H}^{s}(S_{1})=\lbrace g\in H^{s}(\partial\Omega): {\rm supp}(g)\subset \overline{S_{1}}\rbrace$, while $H^{s}(S_{1}):=\lbrace r_{S_{1}}g : g\in H^{s}(\partial\Omega)\rbrace$}, where the notation $r_{S_{1}}g$ is used for the restriction of the function $g$ from $\partial\Omega$ to $S_{1}$.  
 
We will make use of the space
\begin{center}
 $H^{1,0}(\Omega; A):= \lbrace u \in H^{1}(\Omega): Au\in L^{2}(\Omega)\rbrace, $
 \end{center} 
{ see e.g. \cite{Grisvard1985, costabel, traces},} which is a Hilbert space with the norm defined by
\begin{center}
$\parallel u \parallel^{2}_{H^{1,0}(\Omega; A)}:=\parallel u \parallel^{2}_{H^{1}(\Omega)}+\parallel Au  \parallel^{2}_{L^{2}(\Omega)}$.
\end{center}

\paragraph{\bf Traces and conormal derivatives.}
For a scalar function $w\in H^{s}(\Omega^\pm)$, $1/2<s$, the traces $\gamma^{\pm}w\in H^{s-\frac{1}{2}}(\partial\Omega)$  on the Lipschitz boundary $\partial\Omega$ are well defined. 
Moreover, if $1/2<s<3/2$, the corresponding trace operators  $\gamma^{\pm}:=\gamma_{\partial\Omega}^\pm: H^{s}(\Omega^\pm)\to H^{s-\frac{1}{2}}(\partial\Omega)$ are continuous (see, e.g., \cite{McLean2000, traces}).

For $u\in H^{s}(\Omega)$, $s>3/2$, we can define on $\partial\Omega$ the  conormal derivative operator, $T^{\pm}$, in the classical (trace) sense
\begin{equation*}\label{ch4conormal}
T^{\pm}_{x}u := \sum_{i=1}^{3}a(x)\gamma^{\pm}\left( \dfrac{\partial u}{\partial x_{i}}\right)n_{i}^{\pm}(x),
\end{equation*}
where $n^{+}(x)$ is the exterior unit normal vector directed \textit{outwards} the interior domain $\Omega$ at a point $x\in \partial\Omega$. Similarly, $n^{-}(x)$ is the unit normal vector directed \textit{inwards} the interior domain $\Omega$ at a point $x\in \partial\Omega$. 
Sometimes we will also use the notation $T^{\pm}_{x}u$ or $T^{\pm}_{y}u$ to emphasise which respect to which variable we are differentiating. Note that when the variable coefficient $a\equiv 1$, the operator $T^{\pm}$ becomes the classical normal derivative $\partial^{\pm}_{n}$.

Moreover, for any function $u\in H^{1,0}(\Omega; A)$, the \textit{canonical} conormal derivative $T^{\pm}u\in H^{-\frac{1}{2}}(\Omega)$, is well defined, cf. \cite{costabel,McLean2000,traces},
\begin{equation}\label{ch4green1}
\langle T^{\pm}u, w\rangle_{\partial\Omega}:= \pm \int_{\Omega^{\pm}}[(\gamma^{-1}\omega)Au +E(u,\gamma^{-1}w)] dx,\,\, w\in H^{\frac{1}{2}}(\partial\Omega),
\end{equation}
where $\gamma^{-1}: H^{\frac{1}{2}}(\partial\Omega)\to H^{1}(\mathbb{R}^{3})$ is a continuous right inverse to the trace operator whereas the function $E$ is defined as
\begin{equation*}\label{ch4functionalE}
E(u,v)(x):=\sum_{i=1}^{3}a(x)\dfrac{\partial u(x)}{\partial x_{i}}\dfrac{\partial v(x)}{\partial x_{i}},
\end{equation*}
and $\langle\, \cdot \, , \, \cdot \, \rangle_{\partial\Omega}$ represents the $L^{2}-$based dual form on $\partial\Omega$.

We aim to derive boundary-domain integral equation systems for the following \textit{mixed} boundary value problem. Given $f\in L^{2}(\Omega)$, $\phi_{0}\in H^{\frac{1}{2}}(\partial\Omega_D)$ and $\psi_{0}\in H^{-\frac{1}{2}}(\partial\Omega_N)$, we seek a function $u\in H^{1}(\Omega)$ such that 
\begin{subequations}\label{ch4BVP}
\begin{align}
Au&=f,\hspace{1em}\text{in}\hspace{1em}\Omega\label{ch4BVP1};\\
r_{\partial\Omega_{D}}\gamma^{+}u &= \phi_{0},\hspace{1em}\text{on}\hspace{1em} \partial\Omega_{D}\label{ch4BVP2};\\
r_{\partial\Omega_{N}}T^{+}u &=\psi_{0},\hspace{1em}\text{on}\hspace{1em} \partial\Omega_{N};\label{ch4BVP3}
\end{align}
\end{subequations}
where equation \eqref{ch4BVP1} is understood in the weak sense, the Dirichlet condition \eqref{ch4BVP2} is understood in the trace sense, the Neumann condition \eqref{ch4BVP3} is understood in the functional sense \eqref{ch4green1}, { $r_{\partial\Omega_{D}}$ and $r_{\partial\Omega_{N}}$ are restrictions of the functions (or distributions) from $\partial\Omega$ to $\partial\Omega_{D}$ and $\partial\Omega_{N}$, respectively}.

By Lemma 3.4 of \cite{costabel} (cf. also Theorem 3.9 in \cite{traces} for a more general case), the first Green identity holds for any $u\in H^{1,0}(\Omega; A)$ and $v\in H^{1}(\Omega)$,
\begin{equation}\label{ch4green1.1}
\langle T^{\pm}u, \gamma^{+}v\rangle_{\partial\Omega}:= \pm \int_{\Omega}[vAu +E(u,v)] dx.
\end{equation}
The following assertion is well known and can be proved, e.g., using the Lax-Milgram lemma as e.g. in \cite[Theorem 4.11]{steinbach}. 
\begin{theorem}\label{ch4Thomsol}
If the coefficient $a$ satisfies condition \eqref{a-cond}, then the mixed problem \eqref{ch4BVP} has one and only one  solution in $H^1(\Omega)$.
\end{theorem}
 
\section{Parametrices and remainders}

For the differential operator $A$ presented in \eqref{Adef}, we define a parametrix (Levi function) $P(x,y)$ as a function of two (vector) variables $x$ and $y$ such that
 \begin{equation}\label{parametrixdef}
 A_{x}P(x,y) = \delta(x-y)+R(x,y),
 \end{equation}
where the notation $A_{x}$ indicates differentiating with respect to $x$, while the function $R(x,y)$ has at most weak singularity when $x=y$. 
For a given operator $A$, the parametrix is not unique. For example, the parametrix 
\begin{equation*}\label{ch4P2002}
P^y(x,y)=\dfrac{1}{a(y)} P_\Delta(x-y),\hspace{1em}x,y \in \mathbb{R}^{3},
\end{equation*} 
was employed in \cite{localised, CMN-1, mikhailovlipschitz}, for the operator $A$ defined in \eqref{Adef},  where
\begin{equation*}\label{ch4fundsol}
P_{\Delta}(x-y) = \dfrac{-1}{4\pi \vert x-y\vert}
\end{equation*}
is the fundamental solution of the Laplace operator.
The remainder corresponding to the parametrix $P^{y}$ is 
\begin{equation*}
\label{ch43.4} R^y(x,y)
=\sum\limits_{i=1}^{3}\frac{1}{a(y)}\, \frac{\partial a(x)}{\partial x_i} \frac{\partial }{\partial x_i}P_\Delta(x-y)
\,,\;\;\;x,y\in {\mathbb R}^3.
\end{equation*}

{\em In this paper}, for the same operator $A$, we will use another parametrix,  
\begin{align}\label{ch4Px}
P(x,y):=P^x(x,y)=\dfrac{1}{a(x)} P_\Delta(x-y),\hspace{1em}x,y \in \mathbb{R}^{3},
\end{align}
which leads to the corresponding remainder 
\begin{align}\label{ch4remainder}
R(x,y) =R^x(x,y) &= 
-\sum\limits_{i=1}^{3}\dfrac{\partial}{\partial x_{i}}
\left(\frac{1}{a(x)}\dfrac{\partial a(x)}{\partial x_{i}}P_{\Delta}(x,y)\right)\\
& =
-\sum\limits_{i=1}^{3}\dfrac{\partial}{\partial x_{i}}
\left(\dfrac{\partial \ln a(x)}{\partial x_{i}}P_{\Delta}(x,y)\right),\hspace{0.5em}x,y \in \mathbb{R}^{3}.
\nonumber
\end{align}

Note that if the variable coefficient $a$ is smooth enough, then
\[
R^x(x,y),\,R^y(x,y)\in \mathcal{O}(\vert x-y\vert^{-2}) \mbox{ as } x\to y,
\]
i.e., the both remainders $R_x$ and $R_y$ are indeed weakly singular. 

\section{Volume and surface potentials}

For the function $g$ defined on a domain $\Omega_+\subset\mathbb R^n$, e.g., $\rho\in \mathcal D(\overline\Omega)$, the volume parametrix-based Newton-type potential and the remainder potential are respectively defined, for $y\in\mathbb R^3$, as
\begin{align*}
\mathcal{P}\rho(y)&:=\langle P(\cdot,y),g\rangle_{\Omega}=\displaystyle\int_{\Omega} P(x,y)\rho(x)\hspace{0.25em}dx\\
\mathcal{R}\rho(y)&:=\langle R(\cdot,y),g\rangle_{\Omega}=\displaystyle\int_{\Omega} R(x,y)\rho(x)\hspace{0.25em}dx.
\end{align*}
From definitions \eqref{ch4Px}, \eqref{ch4remainder}, the operators $\mathcal P$ and $\mathcal R$ can be expressed in terms the Newtonian potential associated with the Laplace operator, 
\begin{align}
\mathcal{P}\rho&=\mathcal{P}_{\Delta}\left(\dfrac{\rho}{a}\right),\label{ch4relP}\\
\mathcal{R}\rho&=\nabla\cdot\left[\mathcal{P}_{\Delta}(\rho\,\nabla \ln a)\right]-\mathcal{P}_{\Delta}(\rho\,\Delta \ln a).\label{ch4relR}
\end{align}
Relations \eqref{ch4relP} and \eqref{ch4relR} will be used also to determine the operators $\mathcal P$ and $\mathcal R$ also for more general spaces for $\rho$ and to obtain, similar to \cite[Theorem 3.2]{mikhailovlipschitz}, the following  mapping properties of the parametrix-based volume operators from the well-known (cf., e.g., \cite{costabel}) properties of the Newtonian potential associated with the Laplace equation. 
 \begin{theorem}\label{ch4thmUR} 
 Let $s\in \mathbb{R}$. Then, the following operators are continuous,
 \begin{align}
 \mathcal{P}&:\widetilde{H}^{s}(\Omega) \to H^{s+2}(\Omega),\hspace{0.5em} s\in \mathbb{R}, 
 \\
 \mathcal{P}&: H^{s}(\Omega) \to H^{s+2}(\Omega),\hspace{0.5em} -\dfrac{1}{2}<s<\frac{1}{2}, 
 \\
   {\mathcal P}&: L_2(\Omega)\to
 H^{2,0}(\Omega;A),\\
 \mathcal{R}&:\widetilde{H}^{s}(\Omega) \to H^{s+1}(\Omega),\hspace{0.5em} s\in \mathbb{R}, \label{ch4mpvp3}
 \\
 \mathcal{R}&: H^{s}(\Omega) \to H^{s+1}(\Omega),\hspace{0.5em} -\frac{1}{2}<s<\frac{1}{2}\,,
  \label{ch4mpvp4}
 \\
  \mathcal{R}&: H^{1}(\Omega) \to H^{1,0}(\Omega;A).
 \end{align}
Moreover, for $\frac{1}{2}<s<\frac{3}{2}$, the following operators are compact,
 \begin{align*}
 \mathcal{R}&: H^{s}(\Omega) \to H^{s}(\Omega),\\
 r_{S_{1}}\gamma^{+}\mathcal{R}&: H^{s}(\Omega) \to H^{s-\frac{1}{2}}(S_{1}),\\
 r_{S_{1}}T^{+}\mathcal{R}&: H^{s}(\Omega) \to H^{s-\frac{3}{2}}(S_{1}).
 \end{align*}
 \end{theorem}

The parametrix-based single layer and double layer  surface potentials are defined for $y\in\mathbb R^3:y\notin \partial\Omega $, as 
\begin{equation*}\label{ch4SL}
V\rho(y):=-\int_{\partial\Omega} P(x,y)\rho(x)\hspace{0.25em}dS(x),
\end{equation*}
\begin{equation*}\label{ch4DL}
W\rho(y):=-\int_{\partial\Omega} T_{x}^{+}P(x,y)\rho(x)\hspace{0.25em}dS(x).
\end{equation*}
Due to \eqref{ch4Px}, the operators $V$ and $W$ can be also expressed in terms the surface potentials and operators associated with the Laplace operator,
\begin{align}
V\rho &= V_{\Delta}\left(\dfrac{\rho}{a}\right),\label{ch4relSL}\\
W\rho &= W_{\Delta}\rho -V_{\Delta}\left(\rho\frac{\partial \ln a}{\partial n}\right),\label{ch4relDL}
\end{align}
We will use relations \eqref{ch4relSL} and \eqref{ch4relDL}  to determine the operators $V$ and $W$ also for more general spaces for $\rho$ and, using the corresponding properties for the layer potentials based on a fundamental solution, on Lipschitz domains (cf., e.g., \cite{costabel}), to obtain, similar to \cite[Theorem 3.5]{mikhailovlipschitz}, the following mapping and jump properties in Theorems \ref{T3.1s0} and \ref{T3.1sJ}.

\begin{theorem}\label{T3.1s0}
Let  $\Omega$ be a bounded Lipschitz domain.
The following operators are continuous if  $\frac{1}{2}< s< \frac{3}{2}$,
 \begin{eqnarray}
\mu V  &:& H^{s-\frac{3}{2}}(\partial\Omega) \to H^{s}(\mathbb R^n),\quad \forall\ \mu\in\mathcal D(\mathbb R^n);\label{VHs1}\\
r_{\Omega}W  &:&  H^{s-\frac{1}{2}}(\partial\Omega)\to H^{s}(\Omega);\label{WHs1}\\
\mu\,r_{\Omega_-}W  &:&  H^{s-\frac{1}{2}}(\partial\Omega)\to H^{s}(\Omega_-),\quad \forall\ \mu\in\mathcal D(\mathbb R^n)\label{WHs1-};
\\
r_{\Omega}V  &:& H^{-\frac{1}{2}}(\partial\Omega) \to H^{1,0}( {\Omega;A}); \label{VHs1Ga}\\
\mu\,r_{\Omega_-}V  &:& H^{-\frac{1}{2}}(\partial\Omega) \to H^{1,0}({\Omega_-;A}),\quad \forall\ \mu\in\mathcal D(\mathbb R^n);\label{VHs1Ga-}\\
r_{\Omega}W  &:&  H^{\frac{1}{2}}(\partial\Omega)\to  H^{1,0}( {\Omega;A});\label{WHs1Ga}\\
\mu\,r_{\Omega_-}W  &:&  H^{\frac{1}{2}}(\partial\Omega)\to  H^{1,0}( {\Omega_-;A}),\quad \forall\ \mu\in\mathcal D(\mathbb R^n);\label{WHs1Ga-}\\
\gamma^\pm V  &:& H^{s-\frac{3}{2}}(\partial\Omega) \to H^{s-\frac{1}{2}}(\partial\Omega);\label{VHs1g}\\
\gamma^\pm W  &:&  H^{s-\frac{1}{2}}(\partial\Omega)\to H^{s-\frac{1}{2}}(\partial\Omega);\label{WHs1g}\\
T^\pm V  &:& H^{s-\frac{3}{2}}(\partial\Omega) \to H^{s-\frac{3}{2}}(\partial\Omega);\label{VHs1GaT}\\
T^\pm W  &:&  H^{s-\frac{1}{2}}(\partial\Omega)\to  H^{s-\frac{3}{2}}(\partial\Omega).\label{WHs1GaT}
\end{eqnarray}
\end{theorem}
\begin{theorem}\label{T3.1sJ}
Let  $\partial\Omega$ be a compact Lipschitz boundary,  $\frac{1}{2}< s< \frac{3}{2}$, $\varphi\in H^{s-\frac{1}{2}}(\partial\Omega)$ and $\psi\in H^{s-\frac{3}{2}}(\partial\Omega)$. Then
 \begin{eqnarray}
\gamma^+ V\psi -\gamma^- V\psi=0, 
&&
\gamma^+ W\varphi -\gamma^- W\varphi=-\varphi;\label{WHs1gj}\\
T^+ V\psi -T^- V\psi=\psi, 
&&
T^+ W\varphi -T^- W\varphi=-(\partial_n a)\varphi.\label{WHs1GaTj}
\end{eqnarray}
\end{theorem}
Note that the second equation in \eqref{WHs1GaTj} implies that unlike for the classical harmonic potential, the conormal derivative of the parametrix-based double layer potential has a jump.
  
The continuity of operators \eqref{VHs1g}-\eqref{WHs1GaT} in Theorem \ref{T3.1s0} and the first relation in \eqref{WHs1gj} imply the following assertion.
\begin{corollary}\label{T3.3}
Let  $\partial\Omega$ be a compact Lipschitz boundary,  $\frac{1}{2}< s< \frac{3}{2}$.
The following operators are continuous.
 \begin{align}
&\mathcal V:=\gamma^+ V=\gamma^- V  : H^{s-\frac{3}{2}}(\partial\Omega) \to H^{s-\frac{1}{2}}(\partial\Omega);\label{3.11}\\
&\mathcal W:=\frac{1}{2}(\gamma^+ W+\gamma^- W)  :  H^{s-\frac{1}{2}}(\partial\Omega)\to H^{s-\frac{1}{2}}(\partial\Omega);\label{3.12}\\
&\mathcal W':=\frac{1}{2}(T^+V +T^-V)  : H^{s-\frac{3}{2}}(\partial\Omega) \to H^{s-\frac{3}{2}}(\partial\Omega);\label{3.13}\\
&\mathcal L:=\frac{1}{2}(T^+W +T^-W)   :  H^{s-\frac{1}{2}}(\partial\Omega)\to  H^{s-\frac{3}{2}}(\partial\Omega).\label{3.14}
 \end{align}
\end{corollary}
When the boundary and the density $\rho$ are smooth enough, the boundary operators defined in Corollary~\ref{T3.3}  correspond to the  boundary integral (pseudodifferential) operators
of direct surface values of the single layer potential, the double layer potential ${\mathcal W}$, and the co-normal derivatives of the single layer potential ${{\mathcal W}\,^\prime}$ and of the double layer potential, and to the hyper-singular operator, cf. \cite[Eq. (3.6)-(3.8)]{CMN-1} for the parametrix-based potentials on smooth domains, particularly, 
\begin{align}
\mathcal{V}\rho(y)&:=-\int_{\partial\Omega} P(x,y)\rho(x)\hspace{0.25em}dS(x),\nonumber \\
\mathcal{W}\rho(y)&:=-\int_{\partial\Omega} T_{x}P(x,y)\rho(x)\hspace{0.25em}dS(x),\nonumber\\
\mathcal{W'}\rho(y)&:=-\int_{\partial\Omega} T_{y}P(x,y)\rho(x)\hspace{0.25em}dS(x),\nonumber 
\end{align} 
for $y\in \partial\Omega$.
See also \cite[Theorems 7.3, 7.4]{McLean2000} about integral representations on Lipschitz domains of the boundary operators associated with the layer potentials, based on fundamental solutions.

Employing definitions \eqref{3.11}-\eqref{3.14}, the jump properties \eqref{WHs1gj}-\eqref{WHs1GaTj} can be re-written as follows. 
\begin{theorem}\label{ch4thjumps}
For $\psi\in H^{s-\frac{3}{2}}(\partial\Omega)$, and  $\varphi\in H^{s-\frac{1}{2}}(\partial\Omega)$,  $\frac{1}{2}< s< \frac{3}{2}$, 
\begin{align}
&\gamma^\pm V\psi= {\mathcal V}\psi, \hspace{5em}
\gamma^\pm W\varphi= \mp \frac{1}{2}\,\varphi  + {\mathcal W}\varphi;\label{3.8}
\\
&T^\pm V\psi= \pm\frac{1}{2}\,\psi +
{{\mathcal W}\,^\prime}\psi,\
T^\pm W\varphi= \mp \frac{1}{2}\,(\partial_{n} a)\varphi  + {\mathcal L}\varphi.\label{3.11a}
\end{align}
\end{theorem}

By Corollary~\ref{T3.3} and relations \eqref{ch4relSL}-\eqref{ch4relDL}, the operators $\mathcal{V}, \mathcal{W}, \mathcal{W'}$ and $\mathcal{L}$ can be expressed in terms their counterparts (provided with the subscript $\Delta$) associated with the Laplace operator,
\begin{align}
\mathcal{V}\rho &= \mathcal{V}_{\Delta} \left( \dfrac{\rho}{a}\right),\label{ch4relDVSL}\\
\mathcal{W}\rho &= \mathcal{W}_{\Delta}\rho -\mathcal{V}_{\Delta}\left(\rho\frac{\partial \ln a}{\partial n}\right),\label{ch4relDVDL}\\
\mathcal{W}'\rho &= a \mathcal{W'}_{\Delta}\left(\dfrac{\rho}{a}\right),\label{ch4relTSL} \\
\mathcal{L}\rho &= a\mathcal{L}_{\Delta}\rho - a\mathcal{W'}_{\Delta}\left(\rho\frac{\partial \ln a}{\partial n}\right).
\label{ch4relTDL}
\end{align}
 Furthermore, by the Liapunov-Tauber theorem (cf. \cite[Lemma 4.1]{costabel} for the Lipschitz domains), $\mathcal{L}_{\Delta}\rho=T_\Delta^+W_\Delta\rho = T_\Delta^-W_\Delta\rho.$

\begin{theorem}\label{ch4thinvV} Let $S_{1}$ be a non-empty simply connected subset of the Lipschitz surface $\partial\Omega$ with a Lipschitz boundary curve and condition \eqref{a-cond} holds. Then, the operators
\begin{align}
\mathcal{V}&: H^{-\frac{1}{2}}(\partial\Omega) \to H^{\frac{1}{2}}(\partial\Omega),
\label{VdOm}\\
r_{S_{1}}\mathcal{V}&: \widetilde{H}^{-\frac{1}{2}}(S_{1}) \to H^{\frac{1}{2}}(S_{1})
\label{VS1}\end{align}
are continuously invertible.
\end{theorem}
\begin{proof}
We first remark that 
\begin{align}\label{elest1}
\langle \mathcal{V}_\Delta\psi,\psi\rangle_{\partial\Omega}\ge c\|\psi\|_{H^{-1/2}(\partial\Omega)},\quad\forall\psi\in H^{-1/2}(\partial\Omega),
\end{align}
see e.g. \cite[Corollary 8.13]{McLean2000}. 
This evidently gives also
\begin{align}\label{elest2}
\langle \mathcal{V}_\Delta\psi,\psi\rangle_{\partial\Omega}\ge c\|\psi\|_{\widetilde H^{-1/2}(S_1)},\quad\forall\psi\in \widetilde H^{-1/2}(S_1).
\end{align}
By the Lax-Milgram lemma, ellipticity estimates \eqref{elest1},  \eqref{elest2} and the continuity of operators 
$\mathcal{V}_\Delta: H^{-\frac{1}{2}}(\partial\Omega) \rightarrow H^{\frac{1}{2}}(\partial\Omega)$
and
$\mathcal{V}_\Delta: \widetilde H^{-1/2}(S_1) \rightarrow H^{\frac{1}{2}}(S_1)$ imply that
these operators are continuously invertible.
Relation \eqref{ch4relSL} gives $\mathcal{V}g = \mathcal{V}_{\Delta}g^{*}$, where $g^* = g/a$, which leads to 
the invertibility of operators \eqref{VdOm} and \eqref{VS1}.
\end{proof}
 
Let us denote
\begin{align}
\widehat{\mathcal{L}}\rho &:= a\mathcal{L}_{\Delta}\rho.\label{ch4hatL}
\end{align}
Then by \eqref{ch4relTDL} and \eqref{3.11a}, we have,
\begin{align}
T^\pm W\rho= \mp \frac{1}{2}\,(\partial_{n} a)\rho  + \widehat{\mathcal{L}}\rho - a\mathcal{W'}_{\Delta}\left(\rho\frac{\partial \ln a}{\partial n}\right).
\label{3.11aa}
\end{align}
\begin{theorem}\label{ch4thinvL} Let $S_{1}$ be a non-empty simply connected subset of the Lipschitz surface $\partial\Omega$ with a Lipschitz boundary curve and condition \eqref{a-cond} holds. 
Then, the operator
\begin{align}
\label{hatL}
r_{S_{1}}\widehat{\mathcal{L}}: \widetilde{H}^{\frac{1}{2}}(S_{1}) \to H^{-\frac{1}{2}}(S_{1}),
\end{align}
is invertible whilst the operators
\begin{align}
r_{S_{1}}(T^\pm W-\widehat{\mathcal{L}}): \widetilde{H}^{\frac{1}{2}}(S_{1}) \to H^{-\frac{1}{2}}(S_{1})
\label{hatL-}
\end{align}
are compact.
\end{theorem}
\begin{proof}
%
Taking into account the invertibility of the operator 
$r_{S_{1}}\mathcal{L}_\Delta: \widetilde{H}^{\frac{1}{2}}(S_{1}) \to H^{-\frac{1}{2}}(S_{1})$ 
(see e.g.  \cite[Eq. (6.39)]{steinbach} together with the Lax-Milgram lemma), \eqref{ch4hatL} implies the invertibility of operator \eqref{hatL}. 

Now we remark that by \eqref{3.11aa} and the continuity of operator \eqref{3.13}, the operator 
\[ 
r_{S_{1}}(T^\pm W-\widehat{\mathcal{L}}): \widetilde{H}^{-\frac{1}{2}}(S_{1}) \to H^{-\frac{1}{2}}(S_{1})  
\]
is continuous.
Then, the Rellich compact embedding theorem implies the compactness of operators \eqref{hatL-}. 
\end{proof}

\section{Third Green identities and integral relations}

In this section we provide the results similar to the ones in \cite{mikhailovlipschitz} but for our, different, parametrix \eqref{ch4Px}.  
  
Let  $u,v\in H^{1,0}(\Omega;A)$.   Subtracting from the first Green identity \eqref{ch4green1.1} its counterpart with the swapped $u$ and $v$, we arrive at the second Green identity, see e.g. \cite{McLean2000},
\begin{equation}\label{ch4green2}
\displaystyle\int_{\Omega}\left[u\,Av - v\,Au\right]dx
= \langle u, T^{+}v \rangle_{\partial\Omega}-\langle v, T^{+}u\rangle_{\partial\Omega}. 
\end{equation}
Taking now $v(x):=P(x,y)$, and applying \eqref{ch4green2} to the domain $\Omega$ without a small vicinity of $y$, we obtain by the standard limiting procedures (cf. \cite{miranda}) the third Green identity  for any function $u\in H^{1,0}(\Omega;A)$,
\begin{equation}\label{ch4green3}
u+\mathcal{R}u-VT^{+}u+W\gamma^{+}u=\mathcal{P}Au\hspace{1em}\text{in}\hspace{0.2em}\Omega.
\end{equation}

If $u\in H^{1,0}(\Omega; A)$ is a solution of the partial differential equation \eqref{ch4BVP1}, then, from \eqref{ch4green3} we obtain
\begin{equation}\label{ch43GV}
u+\mathcal{R}u-VT^{+}u+W\gamma^{+}u=\mathcal{P}f\hspace{0.5em}\text{in}\hspace{0.2em}\Omega.
\end{equation}
Taking into account the mapping and jump properties of the potentials from Theorems \ref{ch4thmUR}, \ref{T3.1s0} and \ref{ch4thjumps}, we can calculate the traces 
of the both sides of \eqref{ch43GV},
\begin{align}\label{ch43GG}
&\dfrac{1}{2}\gamma^{+}u+\gamma^{+}\mathcal{R}u-\mathcal{V}T^{+}u+\mathcal{W}\gamma^{+}u=\gamma^{+}\mathcal{P}f
\hspace{0.5em}\text{on}\hspace{0.2em}\partial\Omega.
\end{align}
 
For some function $u$ and distributions $f$, $\Psi$ and $\Phi$, we consider a more general, indirect integral relation associated with the third Green identity \eqref{ch43GV},
\begin{equation}\label{ch4G3ind}
u+\mathcal{R}u-V\Psi+W\Phi=\mathcal{P}f\hspace{0.5em} \text{in}\,\,\,\Omega.
\end{equation}
\begin{lemma}\label{ch4lema1}Let $u\in H^{1}(\Omega)$, $f\in L_{2}(\Omega)$, $\Psi\in H^{-\frac{1}{2}}(\partial\Omega)$ and $\Phi\in H^{\frac{1}{2}}(\partial\Omega)$ satisfy the relation \eqref{ch4G3ind}. Then $u$ belongs to $H^{1,0}(\Omega, A)$, it solves the equation 
\begin{equation}\label{ch4lema1.5}
Au=f\hspace{0.5em}{\rm in }\hspace{0.5em}\Omega
\end{equation} 
and the following identity holds true,
\begin{equation}\label{ch4lema1.0}
V(\Psi- T^{+}u) - W(\Phi- \gamma^{+}u) = 0\hspace{0.75em}{\rm in }\hspace{0.5em}\Omega.
\end{equation}
\end{lemma}
\begin{proof}
Since all the potentials in \eqref{ch4G3ind} belong to $H^{1,0}(\Omega;A)$ due to the continuity of operators \eqref{ch4mpvp3}, \eqref{ch4mpvp4},  \eqref{VHs1Ga} and \eqref{WHs1Ga} in Theorems \ref{ch4thmUR} and \ref{T3.1s0}, equation \eqref{ch4G3ind} implies that $u\in H^{1,0}(\Omega;A)$ as well. 

Hence, the third Green identity \eqref{ch4green3} is valid for the function $u$, 
and
we proceed subtracting \eqref{ch4green3} from \eqref{ch4G3ind} to obtain
\begin{equation}\label{ch4lema1.3}
W(\gamma^{+}u-\Phi)-V(T^{+}u-\Psi)=\mathcal{P}(Au-f)\hspace{0.5em}\text{in}\hspace{0.5em}\Omega.
\end{equation}
Let us apply the Laplace operator to both sides of equation \eqref{ch4lema1.3} taking into account relations \eqref{ch4relP}, \eqref{ch4relSL} and \eqref{ch4relDL}. Then, we obtain $Au-f$ in $\Omega$,  
i.e., $u$ solves \eqref{ch4lema1.5}.
Finally, substituting \eqref{ch4lema1.5} into \eqref{ch4lema1.3}, we prove \eqref{ch4lema1.0}.
\end{proof}
 
\begin{lemma}\label{ch4lemma2}
Let $\Psi^{*}\in H^{-\frac{1}{2}}(\partial\Omega)$. If
\begin{equation}\label{ch4lema2i}
V\Psi^{*}= 0\hspace{0.5em}{\rm in }\hspace{0.5em}\Omega,
\end{equation}
then $\Psi^{*} = 0$ on $\partial\Omega$.
\end{lemma}
\begin{proof} Taking the trace of \eqref{ch4lema2i} gives
$\mathcal{V}\Psi^{*} = 0
$ on $\partial\Omega$,
which implies the result due to the invertibility of operator \eqref{VdOm} in Theorem  \ref{ch4thinvV}. 
\end{proof}


\section{BDIE system for the mixed problem}
We aim to obtain a segregated boundary-domain integral equation system for mixed BVP \eqref{ch4BVP}. 
To this end, let $f\in L_{2}(\Omega)$ and the functions $\Phi_{0}\in H^{\frac{1}{2}}(\partial\Omega)$ and $\Psi_{0}\in H^{-\frac{1}{2}}(\partial\Omega)$ be respective continuations of the given boundary data 
$\phi_{0}\in H^{\frac{1}{2}}(\partial\Omega_{D})$ and 
$\psi_{0}\in H^{-\frac{1}{2}}(\partial\Omega_{N})$ to the whole $\partial\Omega$,
i.e., $r_{\partial\Omega_{D}}\Phi_{0}=\phi_{0}$, $r_{\partial\Omega_{N}}\Psi_{0}=\psi_{0}$.
Let us now represent 
\begin{equation}\label{ch4gTrepr}
\gamma^{+}u=\Phi_{0} + \phi,\quad\quad 
T^{+}u = \Psi_{0} +\psi,\quad\text{ on }\,\, \partial\Omega, 
\end{equation}
where  $\phi\in\widetilde{H}^{\frac{1}{2}}(\partial\Omega_{N})$ and $\psi\in\widetilde{H}^{-\frac{1}{2}}(\partial\Omega_{D})$ are unknown boundary functions, which we will further consider as formally independent (segregated) of $u$ in $\Omega$. 

To obtain one of the possible boundary-domain integral equation systems we employ identity \eqref{ch43GV} in the domain $\Omega$, and identity \eqref{ch43GG} on $\partial\Omega$, substituting there relations \eqref{ch4gTrepr}.  Consequently, we obtain the BDIE system (M12) of two equations 
\begin{subequations}\label{ch4SM12vg}
\begin{align}
u+\mathcal{R}u-V\psi+W\phi&=F_{0}\hspace{2em}{\rm in}\hspace{0.5em}\Omega,\label{ch4SM12v}\\
\dfrac{1}{2}\phi+\gamma^{+}\mathcal{R}u-\mathcal{V}\psi+\mathcal{W}\phi&=\gamma^{+}F_{0}-\Phi_{0}\label{ch4SM12g}\hspace{2em}{\rm on}\hspace{0.5em}\partial\Omega,
\end{align}
\end{subequations}
for three unknown functions,  $u$, $\psi$ and $\phi$. Here
\begin{equation}\label{ch4F0term}
F_{0}=\mathcal{P}f+V\Psi_{0}-W\Phi_{0}.
\end{equation}
We remark that $F_{0}$ belongs to the space $H^{1}(\Omega)$ due to the mapping properties of the surface and volume potentials, see Theorems \ref{ch4thmUR} and \ref{T3.1s0}.

\begin{theorem}\label{ch4EqTh}
Let $f\in L_{2}(\Omega)$. Let $\Phi_{0}\in H^{\frac{1}{2}}(\partial\Omega)$ and $\Psi_{0}\in H^{-\frac{1}{2}}(\partial\Omega)$ be some fixed extensions of $\phi_{0}\in H^{\frac{1}{2}}(\partial\Omega_{D})$ and $\psi_{0}\in H^{-\frac{1}{2}}(\partial\Omega_{N})$ respectively. 
\begin{enumerate}
\item[$(i)$] If some $u\in H^{1}(\Omega)$ solves the BVP \eqref{ch4BVP}, then the triple $(u, \psi, \phi )^{\top}\in H^{1}(\Omega)\times\widetilde{H}^{-\frac{1}{2}}(\partial\Omega_{D})\times\widetilde{H}^{\frac{1}{2}}(\partial\Omega_{N})$ where
\begin{equation}\label{ch4eqcond}
\phi=\gamma^{+}u-\Phi_{0},\hspace{4em}\psi=T^{+}u-\Psi_{0}\hspace{2em}{\rm on}\hspace{0.5em}\partial\Omega,
\end{equation}
solves the BDIE system {\rm (M12)}. 

\item[$(ii)$] If a triple $(u, \psi, \phi )^{\top}\in H^{1}(\Omega)\times\widetilde{H}^{-\frac{1}{2}}(\partial\Omega_{D})\times\widetilde{H}^{\frac{1}{2}}(\partial\Omega_{N})$ solves the BDIE system then $u$ solves the BVP and the functions $\psi, \phi$ satisfy \eqref{ch4eqcond}.

\item[$(iii)$] System {\rm (M12)} is uniquely solvable. 
\end{enumerate}
\end{theorem}

\begin{proof}
First, let us prove item $(i)$. Let $u\in H^{1}(\Omega)$ be a solution of the boundary value problem \eqref{ch4BVP}, which implies that $u\in H^{1,0}(\Omega,A)$,  and let $\phi$, $\psi$ be defined by \eqref{ch4eqcond}.
Then, due to \eqref{ch4BVP2} and \eqref{ch4BVP3}, we have \[(\psi,\phi) \in \widetilde{H}^{-\frac{1}{2}}(\partial\Omega_{D})\times\widetilde{H}^{\frac{1}{2}}(\partial\Omega_{N}).\] 

 Then, it immediately follows from the third Green identities \eqref{ch43GV} and \eqref{ch43GG} that the triple $(u,\phi, \psi)$ solves BDIE system $\mathcal{M}^{12}$.

Let us prove now item $(ii)$. Let the triple $(u, \psi,\phi )^{\top}\in H^{1}(\Omega)\times\widetilde{H}^{-\frac{1}{2}}(\partial\Omega_{D})\times\widetilde{H}^{\frac{1}{2}}(\partial\Omega_{N})$ solve the BDIE system. Taking the trace of equation \eqref{ch4SM12v} and subtracting it from equation \eqref{ch4SM12g}, we obtain
the first relation in \eqref{ch4eqcond}. 
Now, restricting it to $\partial\Omega_{D}$, and taking into account that $\phi$ vanishes there as ${\rm supp}\,\phi\subset \overline{\partial\Omega_{N}}$, we obtain that $\phi_{0}=\Phi_{0}=\gamma^{+}u$ on $\partial\Omega_{D}$ and, consequently, the Dirichlet condition \eqref{ch4BVP2} of the BVP  is satisfied.  

We proceed by implementing the Lemma \ref{ch4lema1} to the first equation, \eqref{ch4SM12v}, of system (M12), with $\Psi=\psi + \Psi_{0}$ and $\Phi = \phi + \Phi_{0}$. 
This implies that $u$ belongs to $H^{1,0}(\Omega,A)$, it is a solution of equation \eqref{ch4BVP1} and also the following equality holds,
\begin{equation*}\label{ch4M12a2}
V(\Psi_{0}+\psi - T^{+}u) - W(\Phi_{0} + \phi -\gamma^{+}u) = 0 \text{ in } \Omega.
\end{equation*}
By virtue of the first relation in \eqref{ch4eqcond}, the second term of the previous equation vanishes. Hence,
\begin{equation*}\label{ch4M12a3}
V(\Psi_{0}+\psi - T^{+}u)= 0 \quad \text{ in } \Omega.
\end{equation*}
Now, by virtue of Lemma \ref{ch4lemma2} we obtain
the second relation in \eqref{ch4eqcond}.
Since $\psi$ vanishes on $\partial\Omega_{N}$ and $\Psi_{0}=\psi_{0}$ on $\partial\Omega_{N}$, the second relation in \eqref{ch4eqcond} implies that $u$ satisfies the Neumann condition \eqref{ch4BVP3}. 

Item $(iii)$ immediately follows from the uniqueness of the solution of the mixed boundary value problem, cf. Theorem \ref{ch4Thomsol},
since the zero right-hand side fo the corresponding homogeneous BDIE can be considered as given by $f=0$, $\Psi_0=0$ and $\Phi_0=0$, cf. \eqref{ch4F0term}.
\end{proof}


BDIE system \eqref{ch4SM12v}-\eqref{ch4SM12g} can be written in the matrix notations as 
\begin{equation}\label{M12eq}
\mathcal{M}^{12}\mathcal{X}=\mathcal{F}^{12},
\end{equation}
where $\mathcal{X}$ represents the vector containing the unknowns of the system,
\begin{equation*}
\mathcal{X}=(u,\psi,\phi)^{\top}\in H^{1}(\Omega)\times\widetilde{H}^{-\frac{1}{2}}(\partial\Omega_{D})\times\widetilde{H}^{\frac{1}{2}}(\partial\Omega_{N}),
\end{equation*}
the right hand side vector is \[\mathcal{F}^{12}:= [ F_{0}, \gamma^{+}F_{0} - \Phi_{0} ]^{\top}\in H^{1}(\Omega)\times H^{\frac{1}{2}}(\partial\Omega),\]
and the matrix operator $\mathcal{M}^{12}$ is
\begin{equation}\label{M12mat}
   \mathcal{M}^{12}=
  \left[ {\begin{array}{ccc}
   I+\mathcal{R} & -V & W \\
   \gamma^{+}\mathcal{R} & -\mathcal{V} & \dfrac{1}{2}I + \mathcal{W} 
  \end{array} } \right].
\end{equation}
\begin{theorem}\label{T9}
The operator 
\begin{align}
\label{M12op0}
    \mathcal{M}^{12}&: H^{1,0}(\Omega)\times\widetilde{H}^{-\frac{1}{2}}(\partial\Omega_{D})\times\widetilde{H}^{\frac{1}{2}}(\partial\Omega_{N})\to H^{1,0}(\Omega)\times H^{\frac{1}{2}}(\partial\Omega)\hspace{-1em}
\end{align}
is continuous and continuously invertible.
\end{theorem}

\begin{proof}
The continuity of operator \eqref{M12op0} is implied by the mapping properties of the operators involved in matrix \eqref{M12mat}.

To prove invertibility of operator \eqref{M12op0}, let us consider
BDIE system \eqref{M12eq} with an arbitrary right hand side
$\widetilde{\mathcal F}=\{\widetilde{\mathcal F}_1,\widetilde{\mathcal F}_2\}^\top\in
H^{1,\,0}(\Omega ;\Delta)\times H^{\frac{1}{2}}(\partial\Omega)$. 
From Lemma \ref{A4} in the Appendix, we obtain the representation
\begin{eqnarray*}
&& 
\widetilde{\mathcal F}_1={\mathcal P}\,f_* + V\,\Psi_*-W\,\Phi_*
\;\;\text{in}\;\;  \Omega ,\\
&& 
\widetilde{\mathcal F}_2= \gamma^+\mathcal F^{12}_{*1} -\Phi_*\;\;\text{on}\;\;  \partial\Omega
\end{eqnarray*}
where the triple
\begin{equation}\label{5.33}
(f_*,\Psi_*,\Phi_*)^\top={\widetilde {\mathcal C}}_{*}\,\widetilde{\mathcal F} \in
L_2(\Omega )\times H^{-\frac{1}{2}}(\partial\Omega)\times H^{\frac{1}{2}}(\partial\Omega)
\end{equation}
is unique and the operator
\begin{equation}\label{2.63GG}
{\widetilde {\mathcal C}}_{*}\,:\,H^{1,\,0}(\Omega ;\Delta )\times
H^{\frac{1}{2}}(\partial\Omega)\to L_2(\Omega )\times H^{-\frac{1}{2}}(\partial\Omega)\times
H^{\frac{1}{2}}(\partial\Omega)
\end{equation}
 is linear and continuous.

 Applying the equivalence Theorem~\ref{ch4EqTh} with 
 \begin{equation*} 
f=f_*,\quad \Psi_0=\Psi_*,\quad \Phi_0=\Phi_*,\quad
\psi_0=r_{\partial\Omega_N}\Psi_0,\quad \varphi_0=r_{\partial\Omega_D}\Phi_0,
 \end{equation*} 
 we obtain that the system M12 is uniquely
solvable and its solution is 
 \begin{equation*} 
u=(A^{DN})^{-1}(f_*,r_{\partial\Omega_D}\Phi_*,r_{\partial\Omega_N}\Psi_*)^\top,
\quad \psi=T^+u-\Psi_*,\quad \phi=\gamma^+u-\Phi_*
 \end{equation*}
while $r_{\partial \Omega_N}\psi=0$, $r_{\partial\Omega_D}\phi=0$. Here $(A^{DN})^{-1}$
is the continuous inverse operator to the left-hand-side operator of
the mixed BVP \eqref{ch4BVP}, 
${A}^{DN}:
H^{1,0}(\Omega ;\Delta)\to L_2(\Omega )\times
H^{\frac{1}{2}}(\partial_D\Omega)\times H^{-\frac{1}{2}}(\partial_N\Omega)$. 
Representation
\eqref{5.33}, and continuity of operator  \eqref{2.63GG}
complete the proof of invertibility.
\end{proof}
In the particular case $a(x)=1$ at $x\in \Omega$,  \eqref{ch4operatorA}
becomes the classical Laplace equation, ${\mathcal R}=0$, and BDIE system
\eqref{ch4SM12vg} splits into the Boundary Integral Equation,
BIE,
\begin{equation}\label{4.5TG2Del}
\dfrac{1}{2}\phi-\mathcal{V}_\Delta\psi+\mathcal{W}_\Delta\phi=\gamma^{+}F_{\Delta 0}-\Phi_{0}\hspace{2em}{\rm on}\hspace{0.5em}\partial\Omega,
\end{equation}
where
$F_{\Delta 0}=\mathcal{P}f+V_\Delta\Psi_{0}-W_\Delta\Phi_{0}$, 
and the representation formula for $u$ in terms of $\varphi$ and
$\psi$,
\begin{eqnarray}
\label{4.5GT1Del} u =F_0+ V_{ \Delta \;}\psi - W_{ \Delta
\;}\varphi \qquad &&  \;\;\; \mbox{\rm in}\;\;\;\;   \Omega.
\end{eqnarray}

Then Theorem \ref{ch4EqTh} leads to the following assertion.
\begin{corollary}\label{T4.1TGDel}
Let $a=1$ in $\Omega$, $f\in L_2(\Omega)$, and let  $\Phi_0\in
H^{\frac{1}{2}}(\partial\Omega)$ and $\Psi_0\in H^{-\frac{1}{2}}(\partial\Omega)$ be some extensions of
$\varphi_0\in H^{\frac{1}{2}}(\partial\Omega_D)$ and $\psi_0\in H^{-\frac{1}{2}}(\partial\Omega_N)$,
respectively.

\mbox{\rm (i)} If some $u\in H^1(\Omega)$ solves mixed BVP
\eqref{ch4BVP} in $\Omega$, then the solution is unique,
the couple $(\psi, \varphi)\in {\widetilde H}^{-\frac{1}{2}}(\partial\Omega_D)\times
{\widetilde H}^{\frac{1}{2}}(\partial\Omega_N)$ given by \eqref{ch4eqcond} solves BIE
\eqref{4.5TG2Del}, and $u$ satisfies
\eqref{4.5GT1Del}.

\mbox{\rm (ii)} If a couple $(\psi, \varphi)\in {\widetilde
H}^{-\frac{1}{2}}(\partial\Omega_D)\times {\widetilde H}^{\frac{1}{2}}(\partial\Omega_N)$ solves  BIE 
\eqref{4.5TG2Del}, then  $u$ given by
\eqref{4.5GT1Del} solves BVP \eqref{ch4BVP}  and equations
\eqref{ch4eqcond} hold. Moreover, BIE \eqref{4.5TG2Del}
is uniquely solvable in ${\widetilde H}^{-\frac{1}{2}}(\partial\Omega_D)\times
{\widetilde H}^{\frac{1}{2}}(\partial\Omega_N)$.
\end{corollary}

BIE (\ref{4.5TG2Del}) can be rewritten in the
form
\begin{equation}
\label{4.22TGDel}
  \widehat{\mathcal M}^{12}_\Delta \widehat{\mathcal U}_\Delta=\widehat{\mathcal F}^{12}_\Delta,
\end{equation}
where $\widehat{\mathcal U}^\top_\Delta:=(\psi, \varphi)\in {\widetilde
H}^{-\frac{1}{2}}(\partial\Omega_D)\times {\widetilde H}^{\frac{1}{2}}(\partial\Omega_N)$,
\begin{equation}
\label{4.19TGDelhat} \widehat{\mathcal M}^{12}_\Delta:= \left[
 -{\mathcal V}_\Delta,\
\Big(\frac{1}{2}\,I+{\mathcal W}_\Delta\Big)
\right],\quad
\widehat{\mathcal F}^{12}_\Delta:=\gamma^{+}F_{\Delta 0}-\Phi_{0}.
\end{equation}

  \begin{theorem}\label{GTDeltaInv} 
The operator
\begin{align}\label{hatM}
\widehat{\mathcal M}^{12}_\Delta : {\widetilde H}^{-\frac{1}{2}}(\partial\Omega_D)\times 
{\widetilde H}^{\frac{1}{2}}(\partial\Omega_N)\to H^{-\frac{1}{2}}(\partial\Omega)
\end{align}
is continuous and continuously invertible.
\end{theorem}

\begin{proof} 
The continuity of operator  \eqref{hatM} is implied by the mapping properties of the operators involved in the matrix $\widehat{\mathcal M}^{12}_\Delta$.

A solution of BIE \eqref{4.22TGDel} with an
arbitrary
$\widehat{\mathcal F}^{12}_\Delta\in H^{\frac{1}{2}}(\partial\Omega)$ is delivered by
the couple $(\psi, \varphi)$ satisfying the extended system
\begin{equation}
\label{4.22TGDel2}
  {\mathcal M}^{12}_\Delta {\mathcal U}={\mathcal F}^{12}_{\Delta 0},
\end{equation}
where 
$\mathcal U=(u,\ \psi,\ \varphi)^\top$, ${\mathcal F}^{12}_{\Delta 0}=(0,\ \widehat{\mathcal F}^{12}_{\Delta},)^\top$, 
and
\begin{equation}
\label{4.32TGDel} {\mathcal M}^{12}_\Delta:= \left[
\begin{array}{ccc}
I& -V_\Delta & W_\Delta \\
0  & -{\mathcal V}_\Delta & \displaystyle \frac{1}{2}\,I+\,{\mathcal W}_\Delta
\end{array}
\right]\,.
\end{equation}
The operator 
${\mathcal M}^{12}_\Delta: H^{1,0}(\Omega)\times\widetilde{H}^{-\frac{1}{2}}(\partial\Omega_{D})
\times\widetilde{H}^{\frac{1}{2}}(\partial\Omega_{N})\to H^{1,0}(\Omega)\times H^{\frac{1}{2}}(\partial\Omega)$  
has a continuous inverse due to
Theorem \ref{M12op0} for $a=1$. 
Consequently, operator
\eqref{hatM} has a right continuous inverse,
which is also a two-side inverse due to injectivity of 
operator \eqref{hatM} implied by Corollary \ref{T4.1TGDel}.
\end{proof}

Now we prove the counterpart of Theorem \ref{T9} in wider
spaces.
\begin{theorem}
The operator 
\begin{align}
    \label{M12op}
        \mathcal{M}^{12} &: H^{1}(\Omega)\times\widetilde{H}^{-\frac{1}{2}}(\partial\Omega_{D})\times\widetilde{H}^{\frac{1}{2}}(\partial\Omega_{N})\to H^{1}(\Omega)\times H^{\frac{1}{2}}(\partial\Omega).
\end{align}
is continuous and continuously invertible.
\end{theorem}
\begin{proof}
The continuity of operator  \eqref{M12op} is implied by the mapping properties of the operators involved in matrix \eqref{M12mat}.

Let now $\mathcal{M}_{0}^{12}$ be the matrix operator defined by
\begin{equation*}\label{ch4M012}
  \mathcal{M}_{0}^{12}:=
  \left[ {\begin{array}{ccc}
   I & -V & W_\Delta \\
   0 & -\mathcal{V} & \dfrac{1}{2}I+\mathcal{W}_\Delta \\
  \end{array} } \right].
\end{equation*}
The operator $\mathcal{M}_{0}^{12}$ 
\begin{align}\label{M012}
\mathcal{M}_{0}^{12}: H^{1}(\Omega)\times\widetilde{H}^{-\frac{1}{2}}(\partial\Omega_{D})\times\widetilde{H}^{\frac{1}{2}}(\partial\Omega_{N})\to H^{1}(\Omega)\times H^{\frac{1}{2}}(\partial\Omega).
\end{align}
 is also continuous due to the mapping properties of the operators involved.
 
Let us prove that operator \eqref{M012} is invertible.
First we remark that due to relation \eqref{ch4relDVSL} its second line operator can presented as   
$$
\mathcal{M}_{02}^{12}(\psi, \varphi)^\top
=-{\mathcal V}\psi+ \displaystyle (\frac{1}{2}\,I+\,{\mathcal W}_\Delta)\phi
=\mathcal{M}_{\Delta}^{12}{\rm diag}(\frac{1}{a},1) (\psi, \varphi)^\top.
$$
Then the continuous invertibility of operator \eqref{GTDeltaInv} and  condition \eqref{a-cond} for the coefficient $a$ imply that the operator 
\begin{align*}\label{hatM1}
\widehat{\mathcal M}^{12}_{02}=[-{\mathcal V},\ \frac{1}{2}\,I+\,{\mathcal W}_\Delta] : {\widetilde H}^{-\frac{1}{2}}(\partial\Omega_D)\times 
{\widetilde H}^{\frac{1}{2}}(\partial\Omega_N)\to H^{-\frac{1}{2}}(\partial\Omega)
\end{align*}
is invertible. 
Due to the block-triangular structure of operator $\mathcal{M}_{0}^{12}$ and obvious invertibility of the identity operator, $I$, in $H^{1}(\Omega)$, this, in turn immediately implies invertibility of operator \eqref{M012}.

Further, the operators $V : H^{-\frac{1}{2}}(\partial\Omega) \to H^{1}(\Omega)$ and 
$\mathcal V: H^{-\frac{1}{2}}(\partial\Omega) \to H^{\frac{1}{2}}(\partial\Omega)$ are continuous
by Theorem \ref{T3.1s0} and Corollary \ref{T3.3}. 
Hence, the operators $V : \widetilde H^{\frac{1}{2}}(\partial\Omega_N) \to H^{1}(\Omega)$ and 
$\mathcal V:  \widetilde H^{\frac{1}{2}}(\partial\Omega_N) \to H^{\frac{1}{2}}(\partial\Omega)$
are compact by the Rellich embedding Theorem. 
The operator $\mathcal{R}:H^{1}(\Omega)\to H^{1}(\Omega)$ is also compact by Theorem \ref{ch4thmUR}.
These compactness properties together with representations \eqref{ch4relDL} and \eqref{ch4relDVDL} imply that the operator \eqref{M012}
is a compact perturbation of the operator \eqref{M12op}, which implies its Fredholm property with index one.

Finally, the Fredholm property and the injectivity of operator $\mathcal{M}^{12}$, following from item $(iii)$ of Theorem  \ref{ch4EqTh}, imply the continuous invertibility of operator \eqref{M12op}.
\end{proof}

 
\section{Appendix}
\

We provide below  a simplified version of Lemma 5.5  in \cite{united}. It was proved there for domains with infinitely smooth boundaries but the proof is word-for-word for the Lipschitz domains as well.

\begin{lemma}
\label{Fto-fPsi}  For any function $\mathcal F_0 \in
H^{1,0}(\Omega;\Delta)$, there exists a unique couple
$(f_\Delta ,\Psi_\Delta )={\mathcal C}_0\mathcal F_0 \in L_2(\Omega)\times
{H}^{-\frac{1}{2}}(\partial\Omega)$
 such that
\begin{eqnarray}
\label{4.8HDPsi} \mathcal F_0  &=& \mathcal P_\Delta f_\Delta + V_\Delta\Psi_\Delta  \quad
{\rm in}\,\,\Omega,
\end{eqnarray}
and ${\mathcal C}_0\,:\,H^{1,0}(\Omega;\Delta)\to
L_2(\Omega)\times {H}^{-\frac{1}{2}}(\partial\Omega)$ is a linear bounded operator.
\end{lemma}
Employing Lemma~\ref{Fto-fPsi} for $\mathcal F_{0}=\mathcal F_1+W_\Delta\mathcal F_2\in H^{1,0}(\Omega;\Delta)$, we can easily prove the following assertion (cf. Corollary B.1 in \cite{ayele}.) 
\begin{lemma}
\label{A2}  For any couple 
$(\mathcal F_1,\mathcal F_2)^\top\in
H^{1,0}(\Omega ;\Delta)\times H^{\frac{1}{2}}(\partial\Omega)
$
there exists a unique triple
$$
(f_\Delta,\Psi_\Delta,\Phi_\Delta)^\top=\mathcal C_1\,(\mathcal F_1,\mathcal F_2)^\top \in
L_2(\Omega )\times H^{-\frac{1}{2}}(\partial\Omega)\times H^{\frac{1}{2}}(\partial\Omega)
$$
   such that
\begin{eqnarray*}
&& 
\mathcal F_1={\mathcal P}_\Delta\,f_\Delta + V_\Delta\,\Psi_\Delta-W_\Delta\,\Phi_\Delta
\;\;{\rm in}\;\;  \Omega ,\\
&& 
 \mathcal F_2= \Phi_\Delta \;\;{\rm on}\;\; \partial\Omega.
\end{eqnarray*}
Moreover, the operator
$$ \mathcal C_1\,:\,H^{1,\,0}(\Omega ;\Delta)\times H^{\frac{1}{2}}(\partial\Omega)\to
L_2(\Omega )\times H^{-\frac{1}{2}}(\partial\Omega)\times H^{\frac{1}{2}}(\partial\Omega)
$$
  is linear and continuous.
\end{lemma}
Employing now Lemma~\ref{A2} for $\mathcal F_{1}=\widetilde{\mathcal F}_1\in H^{1,0}(\Omega;\Delta)$ and 
$\mathcal F_{2}=\gamma^+\widetilde{\mathcal F}_1-\widetilde{\mathcal F}_2\in H^{\frac{1}{2}}(\partial\Omega)$, we get the next assertion.
\begin{lemma}
\label{A3}  For any couple 
$(\widetilde{\mathcal F}_1,\widetilde{\mathcal F}_2)^\top\in
H^{1,0}(\Omega ;\Delta)\times H^{\frac{1}{2}}(\partial\Omega)
$
there exists a unique triple
$$
(f_\Delta,\Psi_\Delta,\Phi_\Delta)^\top=\widetilde{\mathcal C}_1\,(\widetilde{\mathcal F}_1,\widetilde{\mathcal F}_2)^\top \in
L_2(\Omega )\times H^{-\frac{1}{2}}(\partial\Omega)\times H^{\frac{1}{2}}(\partial\Omega)
$$
   such that
\begin{eqnarray*}
&& 
\widetilde{\mathcal F}_1={\mathcal P}_\Delta\,f_\Delta + V_\Delta\,\Psi_\Delta-W_\Delta\,\Phi_\Delta
\;\;{\rm in}\;\;  \Omega ,\\
&& 
 \widetilde{\mathcal F}_2= \gamma^+\widetilde{\mathcal F}_1-\Phi_\Delta \;\;{\rm on}\;\; \partial\Omega.
\end{eqnarray*}
Moreover, the operator
$$ \widetilde{\mathcal C}_1\,:\,H^{1,\,0}(\Omega ;\Delta)\times H^{\frac{1}{2}}(\partial\Omega)\to
L_2(\Omega )\times H^{-\frac{1}{2}}(\partial\Omega)\times H^{\frac{1}{2}}(\partial\Omega)
$$
  is linear and continuous.
\end{lemma}
Finally, denoting 
$f_*=af_\Delta$, $\Phi_*=\Phi_\Delta$, $\Psi_*=a\Psi_\Delta-a\Phi_\Delta\frac{\partial \ln a}{\partial n}$, 
where $f_\Delta$, $\Phi_\Delta$ and $\Psi_\Delta$ are the functions and distributions in Lemma~\ref{A3}, it implies the following statement if we take into account relations \eqref{ch4relP}, \eqref{ch4relSL} and \eqref{ch4relDL}.
\begin{lemma}
\label{A4}  For any couple 
$(\widetilde{\mathcal F}_1,\widetilde{\mathcal F}_2)^\top\in
H^{1,0}(\Omega ;\Delta)\times H^{\frac{1}{2}}(\partial\Omega)
$
there exists a unique triple
$$
(f_*,\Psi_*,\Phi_*)^\top=\widetilde{\mathcal C}_*\,(\widetilde{\mathcal F}_1,\widetilde{\mathcal F}_2)^\top \in
L_2(\Omega )\times H^{-\frac{1}{2}}(\partial\Omega)\times H^{\frac{1}{2}}(\partial\Omega)
$$
   such that
\begin{eqnarray*}
&& 
\widetilde{\mathcal F}_1={\mathcal P}\,f_* + V\,\Psi_*-W\,\Phi_*
\;\;{\rm in}\;\;  \Omega ,\\
&& 
 \widetilde{\mathcal F}_2= \gamma^+\widetilde{\mathcal F}_1-\Phi_* \;\;{\rm on}\;\; \partial\Omega.
\end{eqnarray*}
Moreover, the operator
$$ \widetilde{\mathcal C}_*\,:\,H^{1,\,0}(\Omega ;\Delta)\times H^{\frac{1}{2}}(\partial\Omega)\to
L_2(\Omega )\times H^{-\frac{1}{2}}(\partial\Omega)\times H^{\frac{1}{2}}(\partial\Omega)
$$
  is linear and continuous.
\end{lemma} 

\section*{Conclusions}
A new parametrix for the diffusion equation in a continuously non-homogeneous medium (with variable coefficient) with a Lipschitz boundary has been analysed in this paper. Mapping properties of the corresponding parametrix-based surface and volume potentials have been shown in corresponding Sobolev spaces. 

A BDIE system, based on a new parametrix, for the original BVP has been obtained. The equivalence between the BDIE system and the BVP has been shown along with the invertibility of the matrix operator defining the BDIE system. 

Analogous results have been obtained for exterior domains, see \cite{carlos3}, following an approach similar to the one in \cite{exterior}.

A generalisation  to less smooth coefficients and more general PDE right-hand side can be also consider following \cite{mikhailovlipschitz}.  Moreover, these results can be generalised to Bessov spaces as in \cite{miksolandreg}. 

Analysing BDIEs for different parametrices, i.e.  depending on the variable coefficient $a(x)$ or $a(y)$, is crucial to understand the analysis of BDIEs derived with parametrices that depend on the variable coefficient $a(x)$ and $a(y)$ at the same, as it is the case for the Stokes system, see \cite{carlos1,carlos4,carlos5}.


\begin{thebibliography}{99}
%

\bibitem{ayele} Ayele T.G., Mikhailov S.E., Analysis of Two-Operator Boundary-Domain Integral Equations for Variable-Coefficient Mixed BVP, {\it Eurasian Math. J.}, Vol. 2, (2011), \textbf{3}, 20-41. 
%
%
%
%
%

\bibitem{CMN-1}
Chkadua, O., Mikhailov, S.E. and Natroshvili, D.:
Analysis of direct boundary-domain integral equations for a mixed BVP with variable coefficient, I: Equivalence and invertibility.
{\em J. Integral Equations and Appl.} {\bf 21}, 499-543 (2009).

\bibitem{miksolandreg}
Chkadua, O., Mikhailov, S.E. and Natroshvili, D.:
Analysis of direct boundary-domain integral equations for a mixed BVP with variable coefficient, II: Solution regularity and asymptotics.
{\em J. Integral Equations and Appl.}, {\bf 22}, 19-37 (2010).

\bibitem{exterior}Chkadua, O., Mikhailov, S.E. and Natroshvili, D.: Analysis of direct segregated boundary-domain integral equations for variable-coefficient mixed BVPs in exterior domains, {\em Analysis and Applications}, {\bf 11}(4), 1350006 (2013).

\bibitem{costabel}
Costabel, M.:  Boundary integral operators on Lipschitz domains: Elementary results.
{\em SIAM J. Math. Anal.} {\bf 19}, 613-626 (1988).

\bibitem{costste}
Costabel M., Stephan E.P.:{\em An improved boundary element Galerkin method for three dimensional crack problems}
{\em J. Integral Equations Operator Theory} {\bf 10}, 467-507, (1987).

\bibitem{dufera}
Dufera T.T., Mikhailov S.E.: { \em Analysis of Boundary-Domain Integral Equations for Variable-Coefficient Dirichlet BVP in 2D} In: Integral Methods in Science and Engineering: Theoretical and Computational Advances.  C. Constanda and A. Kirsh, eds., Springer (Birkhäuser): Boston, (2015), 163-175.
%
%
%
%
%
%
{
\bibitem{Grisvard1985} Grisvard P., {\em Elliptic Problems in Nonsmooth Domains.} Pitman, Boston (1985).}
\bibitem{numerics}  Grzhibovskis R.,  Mikhailov S.E. and Rjasanow S.: Numerics of boundary-domain integral 
and integro-differential equations for BVP with variable coefficient in 3D, {\em Computational Mechanics}, {\bf 51}, 495-503 (2013).
%
%
\bibitem{hsiao}
Hsiao G.C. and Wendland W.L.:
{\em Boundary Integral Equations.}
Springer, Berlin (2008).

%
%
%
%
%
%
%
%
\bibitem{lions}
Lions J.L. and Magenes E.:
{\em Non-Homogeneous Boundary Value Problems and Applications.}
Springer (1973).
%

\bibitem{McLean2000}
McLean W.:
{\em Strongly Elliptic Systems and Boundary Integral Equations.}
Cambridge University Press (2000).
%

\bibitem{localised} Mikhailov S.E.: Localized boundary-domain integral formulations for problems
with variable coefficients, \emph{Engineering Analysis with Boundary
Elements}, \textbf{26} (2002) 681-690.
  
\bibitem{united}  Mikhailov S.E.: Analysis of united boundary-domain integro-differential and integral equations for a mixed BVP with variable coefficient, \emph{Math. Methods Appl. Sci.}, \textbf{29} (2006) 715-739.

\bibitem{traces}
Mikhailov S.E.:
Traces, extensions and co-normal derivatives for elliptic systems on Lipschitz domains.  {\em J. Math. Anal. and Appl.}, {\bf 378}, (2011) 324-342. 


\bibitem{mikhailovlipschitz} Mikhailov S.E.  Analysis of Segregated Boundary-Domain Integral Equations for BVPs with Non-smooth Coefficient on Lipschitz Domains, {{\em Boundary Value Problems}, {\bf 2018}:87, 1-52 (2018).}
  

{
\bibitem{MikMoh2012}
Mikhailov, S.E.,
Mohamed, N.A.:
Numerical solution and spectrum of boundary-domain integral equation
  for the {N}eumann {BVP} with variable coefficient.
{\em Internat. J. Comput. Math.}
{\bf 89},
(2012)
{1488--1503}.
doi:{10.1080/00207160.2012.679733}
}

\bibitem{localised2} Mikhailov S.E., Nakhova I.S.  Mesh-based numerical implementation
of the localized boundary-domain integral-equation
method to a variable-coefficient Neumann problem. \textit{J. Engineering Math.}, \textbf{51} (2005): 251–259.
  
\bibitem{carlos1} Mikhailov S.E., Portillo C.F.:  BDIE System to the Mixed BVP for the Stokes Equations with Variable Viscosity, {\em Integral Methods in Science and Engineering: Theoretical and Computational Advances.}  C. Constanda and A. Kirsh, eds., Springer (Birkhäuser): Boston (2015).

\bibitem{carlos2} Mikhailov S.E.,  Portillo C.F.: {\em A New Family o Boundary-Domain Integral Equations for a Mixed Elliptic BVP with Variable Coefficient}, in \textit{Proceedings of the 10th UK Conference on Boundary Integral Methods}, (Brighton University Press, 2015).

\bibitem{carlos3} Mikhailov S.E.,  Portillo C.F.: {\em A New Family of Boundary-Domain Integral Equations for the Mixed Exterior Stationary Heat Transfer Problem with Variable Coefficient}, in \textit{Integral Methods in Science and Engineering, Volume 1}, (Birkhäuser, Cham, 2017).

\bibitem{carlos4} Mikhailov S.E.,  Portillo C.F.: 
{\em BDIES for the compressible Stokes system with variable viscosity mixed BVP in bounded domains}, in \textit{Proceedings of the 11th UK Conference on Boundary Integral Methods (UKBIM 2017)}, (Nottingham Trent University).

\bibitem{carlos5} Mikhailov S.E.,  Portillo C.F.: 
{\em Analysis of Boundary-Domain Integral Equations to the Mixed BVP for a Compressible Stokes System with Variable Viscosity}, arXiv:1805.00235 (2018).
%
\bibitem{miranda}  Miranda C.: {\em Partial Differential Equations of Elliptic Type} 2nd edn. Springer, (1970).

%
%
%
%
%
%
%
%
%
%
{
\bibitem{RS2013} Ravnik, J.,  \v{S}kerget, L.: {A gradient free integral equation for diffusion–convection equation
with variable coefficient and velocity}. \emph{Engineering Analysis with Boundary Elements}. \textbf{37} (2013), 683--690.

\bibitem{RS2014} Ravnik, J.,  \v{S}kerget, L.: {Integral equation formulation of an unsteady diffusion–convection equation with variable coefficient and velocity}. \emph{Computers and Mathematics with Applications}. \textbf{66} (2014), 2477--2488.

\bibitem{SHK1999}  \v{S}kerget, L, Hrber\v{s}ek, M., Kuhn, G.: {Computational fluid dynamics by boundary-domain integral method}. {\em Internat. J. Numerical Methods in Engineering}. \textbf{46} (1999), 1291--1311.
}
\bibitem{sladek} Sladek J., Sladek V., Zhang Ch. (2005) Local integro-differential
equations with domain elements for the numerical solution of partial
differential equations with variable coefficients. \textit{Journal of Engineering Mathematics}
\textbf{51}, (2005), 261–282.

\bibitem{steinbach} Steinbach O.: {\em Numerical Approximation Methods for Elliptic Boundary Value Problems}. Springer (2007).


\end{thebibliography}
\end{document}